\documentclass[12pt, reqno]{amsart}
\usepackage[utf8]{inputenc}
\usepackage[english]{babel}
 
\usepackage[square,sort,comma,numbers]{natbib}

\usepackage{fancyhdr}
\usepackage{enumitem}

\usepackage{float}

\usepackage{graphicx}
\usepackage{tikz}
\usetikzlibrary{arrows}
\usetikzlibrary{matrix,arrows,decorations.pathmorphing}

\usepackage{graphicx}
\usepackage{tikz}
\usetikzlibrary{arrows}
\usetikzlibrary{matrix,arrows,decorations.pathmorphing}

\usepackage{amsmath}
\usepackage{amssymb}
\usepackage{amsthm}

\usepackage{stmaryrd}
\usepackage[all]{xy}
\usepackage[mathcal]{eucal}
\usepackage{verbatim} 
\usepackage{fullpage} 
\usepackage{wrapfig}
\usepackage{lipsum}
\usepackage[all]{xy}
\title{The number of fiberings of a surface bundle over a surface}

\author{Lei Chen}
\email{chenlei1991919@gmail.com}

\numberwithin{equation}{section}
\numberwithin{figure}{section}

\theoremstyle{plain}
\newtheorem{thm}{Theorem}[section]

\newtheorem{lem}[thm]{Lemma}

\newtheorem{claim}[thm]{Claim}
\newtheorem{rmk}[thm]{Remark}
\newtheorem{defn}[thm]{Definition}

\newtheorem{qu}[thm]{Question}

\theoremstyle{remark}

\begin{document}

\begin{abstract}
	For a closed manifold $M$, let SFib$(M)$ be the number of ways that $M$ can be realized as a surface bundle, up to $\pi_1$-fiberwise diffeomorphism. In this paper we consider the case of dim$(M)=4$. We give the first computation of SFib$(M)$ where $1<\text{SFib}(M)<\infty$ but $M$ is not a product. In particular, we prove SFib$(M)=2$ for the Atiyah--Kodaira manifold and any finite cover of a trivial surface bundle. We also give an example where SFib$(M)=4$.
	\end{abstract}
	\maketitle
\section{Introduction}
Let $M$ be a closed manifold and for any $g>1$ let $S_g$ denote a closed, connected, orientable surface of genus $g$. We will call the following number the \emph{surface-fibering number} of $M$:

\newcommand{\pctext}[2]{\text{\parbox{#1}{\centering #2}}}

\begin{alignat}{2}
      \label{correspondence}
 \pctext{0.7in}  {SFib$(M)$=}    \#                                                 
    \Biggl\{  \pctext{2in}{$S_g\to M \to B$  a surface bundle $:$ $g>1$ and B a closed manifold} \Biggr\}    \bigg/\sim,
      \end{alignat}
where two fiberings of $M$ are equivalent if and only if they are \emph{$\pi_1$-fiberwise diffeomorphic}; i.e. if the fundamental groups of their fibers are the same subgroups of $\pi_1(M)$. See Section 2 below for more details. The equivalence relation of $\pi_1$-fiberwise diffeomorphism is more natural than fiberwise diffeomorphism algebraically because it classifies fiberings based on how $\pi_1(M)$ can be represented as an extension by $\pi_1(S_g)$ for some $g>1$. Also $\pi_1$-fiberwise diffeomorphism is finer than fiberwise diffeomorphism; so SFib$(M)$ is an upper bound for the number of fiberings up to fiberwise diffeomorphism.

In the case dim$(M)=3$, Thurston \cite{Thurnorm} classified all possible surface bundle structures on a fixed $M$ using the Thurston norm. His theory implies that if $M$ is a surface bundle over $S^1$, then SFib$(M)=\infty$ if and only if dim$\left( H^1(M;\mathbb{Q})\right)>1$. Further, the nonzero $\mathbb{Z}$-points in the so-called `fibered cone' of $H^1(M;\mathbb{R})$ up to scalar multiplication are in one-to-one correspondence with distinct fiberings.

In this paper we study the case where dim$(M)=4$; in other words, the case where $M$ is a surface bundle over a surface. When the Euler characteristic $\chi(M)>0$, F.E.A. Johnson \cite{FEAJohnson} proved that SFib$(M)<\infty$. He also obtained an upper bound for SFib$(M)$ depending only on $\chi(M)$. For any $N>1$, Salter \cite{NickFibering} constructed an example $M_N$ such that SFib$(M_N)>N$. His work does not give the exact value of SFib$(M_N)$ for any $N$. Salter \cite{salter2015,NickFibering} proved that if the monodromy of a nontrivial bundle $S_g\to M\to B$ is in the Johnson kernel, then SFib$(M)=1$. He also proved that if $H^1(M;\mathbb{Q})\cong H^1(B;\mathbb{Q})$ then SFib$(M)=1$.

	One beautiful example of a multi-fibered 4-manifold is the \emph{Atiyah--Kodaira manifold} $M_{AK}$; see Atiyah's paper \cite{atiyah2015signature}, Kodaira's paper \cite{MR0216521} or Section 3 below for the construction. It follows from the construction that $M_{AK}$ has at least two different fiberings:
		\[
	S_6\to M_{AK} \to S_{129} \text{ }\text{ }\text{ }\text{ and }\text{ }\text{ }\text{ }
	S_{321}\to M_{AK}\to S_{3.}\]
	It is natural to ask if there are any other fiberings. Our first theorem answers this question in the negative.
\begin{thm}[\boldmath \bf Surface-fibering number of $M_{AK}$]
The Atiyah--Kodaira manifold has precisely two fiberings up to $\pi_1$-fiberwise diffeomorphism; that is
{\normalfont 
SFib$(M_{AK})=2$.} In particular, $M_{AK}$ has precisely two fiberings up to fiberwise diffeomorphism.
\label{main2}
\end{thm}
As mentioned above, fiberwise diffeomorphism is implied by $\pi_1$-fiberwise diffeomorphism. In particular, $M_{AK}$ has two fiberings up to fiberwise diffeomorphism because the two fiberings of $M_{AK}$ have different fibers and thus are clearly not fiberwise diffeomorphic. While $M_{AK}$ has been well-studied in the last 50 years by Atiyah, Hirzebruch, Kodaira and many others, we will show that there are choices involved in the construction, which are parametrized by elements in $H^1(S_{129}\times S_3;\mathbb{Z})$. See Section 3 for details. At the end of Section 3.1, we will pose the question of whether the different Atiyah--Kodaira manifolds we construct are diffeomorphic to one another as smooth manifolds.

Denote the genus of a closed oriented surface $S$ by $g(S)$. We can also compute the surface-fibering number of a finite cover over a product $B\times F$ where $B$ and $F$ are two surfaces with $g(B)>1$ and $g(F)>1$.

\begin{thm}[\bf Finite cover of a trivial bundle]
Let $E$ be a regular finite cover of a trivial bundle $B\times F$ where $B$ and $F$ are two surfaces with $g(B)>1$ and $g(F)>1$. Then \normalfont{SFib}$(E)=2$.
\label{thm3}
\end{thm}

Salter \cite{NickFibering} constructed a certain 4-manifold $M_S$ by performing a section sum of two copies of $S_g\times S_g$; see Section 6 for the construction. He provided 4 distinct fiberings of $M_S$; so SFib$(M_S)\ge 4$. Our next theorem classifies the fiberings of $M_S$.

\begin{thm}[\bf Salter's 4-fibering example]
Salter's example $M_S$ has precisely 4 fiberings up to $\pi_1$-fiberwise diffeomorphism; that is
\normalfont{SFib}$(M_S)=4$.
\label{thm5}
\end{thm}
Unlike the Atiyah--Kodaira example, the four fiberings of $M_S$ are actually fiberwise diffeomorphic to one another but not $\pi_1$-fiberwise diffeomorphic to one another.

All the known examples have SFib$(M)$ a power of 2. We conjecture that all the examples that Salter built in \cite{NickFibering} have SFib$(M)$ a power of 2. Therefore, we ask the following question.
\begin{qu}[\bf 3 fiberings construction]
Is there a surface bundle over a surface with total space $M$ such that \normalfont{SFib}$(M)$ is not a power of $2$?
\end{qu}

\noindent
{\large\bf Acknowledgements}

The author would like to thank Nick Salter, Ben O'Connor and Nir Gadish for their discussions related to this topic and for correcting the paper. She is grateful to anonymous referees, to Justin Lanier and to Dan Margalit for many helpful suggestions. She would also like to extend her warmest thanks to Benson Farb for asking the question, for his extensive comments and for his invaluable support from start to finish.

\section{Definition of equivalent fiberings and a criterion for two fiberings}
In this section we will introduce the definition of $\pi_1$-fiberwise diffeomorphism, which is the equivalence relation we use in defining fibering numbers. We will also give a cohomological criterion for a 4-manifold $M$ such that SFib$(M)=2$. In this article, we only discuss surface bundles rather than general fiber bundles. Thus we will use ``fiberings" to replace ``surface-fiberings." When we talk about fundamental groups in this paper, we omit the base point.

\begin{defn}[\boldmath \bf $\pi_1$-fiberwise diffeomorphism]
Given any closed manifold $M$, two fiberings $F_1\to M\xrightarrow{p_1} B_1$ and $F_2\to M\xrightarrow{p_2} B_2$ are \textbf{\boldmath $\pi_1$-fiberwise diffeomorphic} if they satisfy the following conditions:

1) There exists the following diagram:
\[
\xymatrix{
M\ar[r]^{a}\ar[d]^{p_1} & M\ar[d]^{p_2}\\
B_1\ar[r]^{b} & B_2}
\] 
where $a,b$ are both diffeomorphisms.

2) Let $a_*:\pi_1(M)\to \pi_1(M)$ be the induced map on the fundamental groups. We have that $a_*(\pi_1(F_1))=\pi_1(F_1)$.
\label{pid}
\end{defn}
As equivalence relations, $\pi_1$-fiberwise diffeomorphism is finer than fiberwise diffeomorphism or bundle diffeomorphism, where we do not assume the second condition in Definition \ref{pid}. In other words, the numbering of fiberings of $M$ up to fiberwise diffeomorphism is at most SFib$(M)$. To further classify the fiberings up to fiberwise diffeomorphism, we only need to check all the equivalence classes under $\pi_1$-fiberwise diffeomorphism. We use $\pi_1$-fiberwise diffeomorphism because it is more natural on the group-theoretic level. Two fiberings $F_1\to M\xrightarrow{p_1} B_1$ and $F_2\to M\xrightarrow{p_2} B_2$ are $\pi_1$-fiberwise diffeomorphic if and only if $\pi_1(F_1)$ and $\pi_1(F_2)$ are the same subgroups in $\pi_1(M)$. From now on, we call two fiberings equivalent if they are $\pi_1$-fiberwise diffeomorphic. We have the following lemma of Salter \cite[Lemma 3.3]{NickFibering}.

\begin{lem}
Given any closed 4-manifold $M$, if there are two fiberings $M\xrightarrow{p_1} B_1$ and $M\xrightarrow{p_2} B_2$ that are not equivalent, then $p_1^*(H^1(B_1;\mathbb{Q})) \cap p_2^*(H^1(B_2;\mathbb{Q}))=\{0\}$.
\label{nick}
\end{lem}
The following lemma is a cohomological criterion for a 4-manifold $M$ such that SFib$(M)=2$.

\begin{lem}[\boldmath  \bf A Criterion for SFib$(M)=2$]
Let $S_{h_1}\to M\xrightarrow{p_1} S_{g_1}$ and $S_{h_2}\to M\xrightarrow{p_2} S_{g_2}$ be two surface bundles over a surface where $h_1,g_1,h_2,g_2>1$ and $p_1$ is not equivalent to $p_2$. Let $(p_1,p_2):M\to S_{g_1}\times S_{g_2}$. If 
\[
(p_1,p_2)^*:H^1(S_{g_1}\times S_{g_2};\mathbb{Q})\cong p_1^*H^1(S_{g_1};\mathbb{Q})\oplus p_2^*H^1(S_{g_2};\mathbb{Q})\cong H^1(M;\mathbb{Q})
\]
 and if 
 \[(p_1,p_2)^*:H^2(S_{g_1}\times S_{g_2};\mathbb{Q})\to H^2(M;\mathbb{Q}) \text{   is injective, }
 \]
then \normalfont{SFib}$(M)=2$.
\label{main1}
\end{lem}

\begin{proof}
Suppose there exists a third fibering $F\to M\xrightarrow{p} B$ such that $p$ is not equivalent to $p_1$ or $p_2$. By Lemma \ref{nick}, for any nonzero element $x\in H^1(B;\mathbb{Q})$, we have that $p^*(x)\notin p_1^*H^1(S_{g_1};\mathbb{Q})$ and $p^*(x)\notin p_2^*H^1(S_{g_2};\mathbb{Q})$. Therefore, there exists $a\neq 0\in p_1^*H^1(S_{g_1};\mathbb{Q})$ and $b \neq 0 \in p_2^*H^1(S_{g_2};\mathbb{Q})$ such that
\[
p^*(x)=a+b \in p_1^*H^1(S_{g_1};\mathbb{Q})\oplus p_2^*H^1(S_{g_2};\mathbb{Q})\cong H^1(M;\mathbb{Q}).
\]
Since $\chi(M)>0$ and $\chi(F)<0$, we have $\chi(B)<0$ implying $g(B)>1$. Therefore we have another element $y\neq 0\in H^1(B;\mathbb{Q})$ not a multiple of $x$ but satisfying that 
 \[
 x\smile y=0\in H^2(B;\mathbb{Q}).
 \] 
 Suppose that 
 \[
p^*(y)=c+d.
\]
Since $x\smile y=0$,
\[
(a+b)\smile (c+d)=0\in  (p_1,p_2)^*H^2(S_{g_1}\times S_{g_2};\mathbb{Q})\subset H^2(M;\mathbb{Q}).
\]
By the K\"unneth formula
\[
H^2(S_{g_1}\times S_{g_2};\mathbb{Q})\cong H^2(S_{g_1};\mathbb{Q})\oplus \left[H^1(S_{g_1};\mathbb{Q})\otimes H^1(S_{g_2};\mathbb{Q})\right]\oplus H^2(S_{g_2};\mathbb{Q}),
\]
$a\smile d+b\smile c=0$. By skew-commutativity of cup product, $a\smile d=c\smile b$. By the property of the tensor product of vector spaces, the only possibility is that $c=ka$ and $d=kb$ for some $k\in \mathbb{Q}$. In this case, $y$ is a multiple of $x$, which contradicts the fact that $y$ is not a multiple of $x$. The result follows.
\end{proof}

\section{\boldmath Description of $M_{AK}$ and the uniqueness problem}
In this section, we will describe the Atiyah--Kodaira manifold $M_{AK}$ and also its monodromy representation. 
While $M_{AK}$ has been vastly studied in the last 50 years, we will show below that there are choices involved in the construction, which are parametrized by elements in a cohomology group. At the end, we will pose the question of whether the different Atiyah--Kodaira manifold determine diffeomorphic manifolds.
\subsection{\boldmath The geometric construction of $M_{AK}$}
\indent
We now construct the Atiyah--Kodaira manifold $M_{AK}$, following Morita \cite[Chapter 4.3]{Mor}. Let $S_3$ be a surface of genus $3$ and let $\tau$ be a free $\mathbb{Z}/2\mathbb{Z}$-action on $S_3$ as Figure \ref{figure10}. The trivial bundle $S_3\times S_3$ has 2 sections: $\Gamma_{id}$ the graph of the identity and $\Gamma_\tau$ the graph of $\tau$. 

\begin{figure}[H]
\centering
 \includegraphics[scale=0.55]{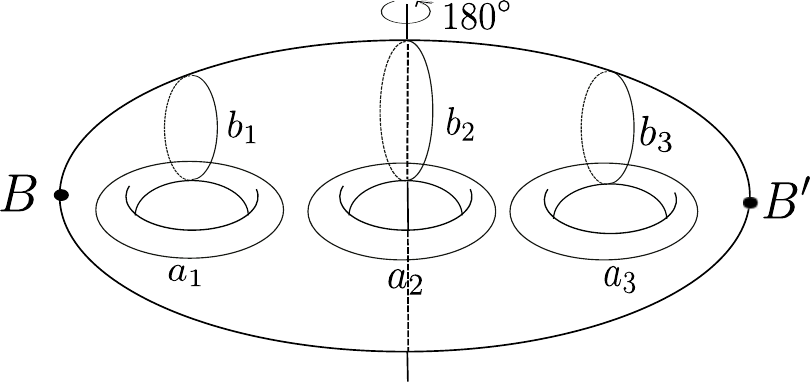}
 \caption{Involution $\tau$}
  \label{figure10}
\end{figure}

Since the action is free, the two sections are disjoint. The kernel of the surjective homomorphism $\pi_1(S_3)\to H_1(S_3;\mathbb{Z}/2)$ gives a finite cover $i:S_{129}\to S_3$. We have the following exact sequence:
\[
1\to \pi_1(S_{129})\xrightarrow{i_*}\pi_1(S_3)\to H_1(S_3;\mathbb{Z}/2)\to 1.
\]
The pull-back surface bundle $i^*(S_3\times S_3)\cong S_{129}\times S_3$ also has 2 sections $S_i=i^*(\Gamma_{id})$ and $S_\tau=i^*(\Gamma_\tau)$. We have $S_i=\text{graph}(i)$ and $S_{\tau}=\text{graph}(\tau\circ i)$. The plan now is to characterize $\mathbb{Z}/2$-branched covers over $S_{129}\times S_3$ with branch locus $S_i\cup S_{\tau}$. We begin by computing the Poincar\'e dual of the homology class $[S_i]+[S_{\tau}]$. The K\"unneth formula gives us the following:
\[
H_2(S_{129}\times S_3;\mathbb{Z}/2)\cong H_2(S_{129};\mathbb{Z}/2)\oplus\left[H_1(S_{129};\mathbb{Z}/2)\otimes H_1(S_3;\mathbb{Z}/2)\right]\oplus H_2(S_3;\mathbb{Z}/2).
\]
Let $[S_{129}]$ and $[S_3]$ be the fundamental classes of $H_2(S_{129};\mathbb{Z}/2)$ and $H_2(S_3;\mathbb{Z}/2)$. Pick points $p_0\in S_{129}$ and $q_0\in S_3$. Let $e_1:S_{129}\to S_{129}\times S_3$ and $e_2:S_{3}\to S_{129}\times S_3$ be maps $e_1(x)=(x,q_0)$ and $e_2(y)=(p_0,y)$. By the computation in \cite[Chapter 11]{MR0440554} and the fact that $i_*$ induces zero map on $H_1(\text{---};\mathbb{Z}/2)$, we have $[S_i]=e_{1*}[S_{129}] \text{ and } [S_{\tau}]=e_{1*}[S_{129}]\in H_2(S_{129}\times S_3;\mathbb{Z}/2)$. Therefore 
\[
[S_i]+[S_{\tau}]=e_{1*}[S_{129}]+e_{1*}[S_{129}]=0\in H_2(S_{129}\times S_3;\mathbb{Z}/2).
\]
Denote the Poincar\'e dual of $[S_i]+[S_{\tau}]$ by $PD([S_i]+[S_{\tau}])$. By Poincar\'e duality
 \[
 PD([S_i]+[S_{\tau}])=0\in H^2(S_{129}\times S_3;\mathbb{Z}/2).
 \]
Let $M:=S_{129}\times S_3-S_i-S_{\tau}$. We have the long exact sequence of cohomology of the relative pair $(S_{129}\times S_3,M)$:
\begin{equation}
H^1(S_{129}\times S_3,M;\mathbb{Z}/2)\to H^1(S_{129}\times S_3;\mathbb{Z}/2)\to H^1(M;\mathbb{Z}/2)\xrightarrow{\phi} H^2(S_{129}\times S_3,M;\mathbb{Z}/2)\xrightarrow{T} H^2(S_{129}\times S_3;\mathbb{Z}/2).
\label{phi}
\end{equation}
\indent
By the Thom isomorphism theorem, we have 
\[H^1(S_{129}\times S_3,M;\mathbb{Z}/2)=0\]
 and 
\[
H^2(S_{129}\times S_3,M;\mathbb{Z}/2)\cong \mathbb{Z}/2\oplus \mathbb{Z}/2.
\] 
Let $T$ and $\phi$ be the homomorphisms in the exact sequence (\ref{phi}). Now $T(1,0)=PD[S_i]$ and $T(0,1)=PD[S_{\tau}]$. Therefore 
\[
T(1,1)=0\in H^2(S_{129}\times S_3;\mathbb{Z}/2).
\]
So $\phi^{-1}(1,1)$ is not empty in $H^1(M;\mathbb{Z}/2)$. By the isomorphism
\[
\text{Hom}(\pi(M),\mathbb{Z}/2)\cong H^1(M;\mathbb{Z}/2),
\]
we have that $H^1(M;\mathbb{Z}/2)$ classifies $\mathbb{Z}/2$-covers of $M$. Therefore $\phi^{-1}(1,1)$ classifies the $\mathbb{Z}/2$-branched covers of $S_{129}\times S_3$ with branch locus $S_i\cup S_{\tau}$. Let $M_{AK}$ be one of them. These branched covers are characterized by a subset of $H^1(M;\mathbb{Z}/2)$, which is an affine space over $H^1(S_{129}\times S_3;\mathbb{Z}/2)$ by the exact sequence \eqref{phi}. Later we will analyze how an element of $H^1(S_{129}\times S_3;\mathbb{Z}/2)$ affects the monodromy. We also pose a question about the Atiyah--Kodaira construction.
\begin{qu}[\bf Uniqueness of Atiyah--Kodaira example]
After fixing the trivial bundle $S_{129}\times S_3$ and the two sections $S_i$ and $S_{\tau}$, there are many choices of branched covers of $S_{129}\times S_3$ with branch locus $S_i\cup S_{\tau}$. Are the different branched covers diffeomorphic as smooth manifolds?
\end{qu}
\subsection{\boldmath The monodromy description of $M_{AK}$}
In this subsection, we provide a second construction of $M_{AK}$ from the point of view of the monodromy representation. As is known, for $g>1$ the monodromy representation determines an $S_g$-bundle uniquely; see e.g. the book by Farb and Margalit \cite[Chapter 5.6]{BensonMargalit}.

Let $S_{g,n}$ be a genus $g$ surface with $n$ punctures. Let PMod$_{g,n}$ (resp. Mod$_{g,n}$) be the pure mapping class group of $S_{g,n}$, i.e. the group of isotopy classes of orientation-preserving diffeomorphisms of $S_g$ that fix $n$ points individually (resp. as a set). We omit $n$ when $n=0$. Let PConf$_n(S)$ \emph{the pure configuration space} of a surface $S$ be the space of ordered $n$-tuples of distinct points on $S$. We have a \emph{generalized Birman exact sequence} as the following \cite[Theorem 9.1]{BensonMargalit}:
\[
1\to \pi_1(\text{PConf}_n(S_g))\xrightarrow{\text{Push}} \text{PMod}_{g,n}\to \text{Mod}_g\to 1.
\]
The two disjoint sections of the bundle $S_3\times S_3$ give us a map $(id,\tau):S_3\to \text{PConf}_2(S_3)$, and hence a monodromy representation:
\[
\rho_\tau:\pi_1(S_3)\to \pi_1(\text{PConf}_2(S_g))\xrightarrow{\text{Push}} \text{PMod}_{3,2}.
\] 
Let $B\in S_3$ and $B'=\tau(B)$. The $\mathbb{Z}/2$-branched covers of $S_3$ with branch points $B$ and $B'$ are parametrized by a subset of  $H^1(S_{3,2};\mathbb{Z}/2)$. Pick any $\mathbb{Z}/2$-branched cover $\pi: S_{6}\to S_{3}$ with a deck transformation $\sigma$ and branch points $\{B,B'\}$.

\begin{figure}[H]
\centering
 \includegraphics[scale=0.55]{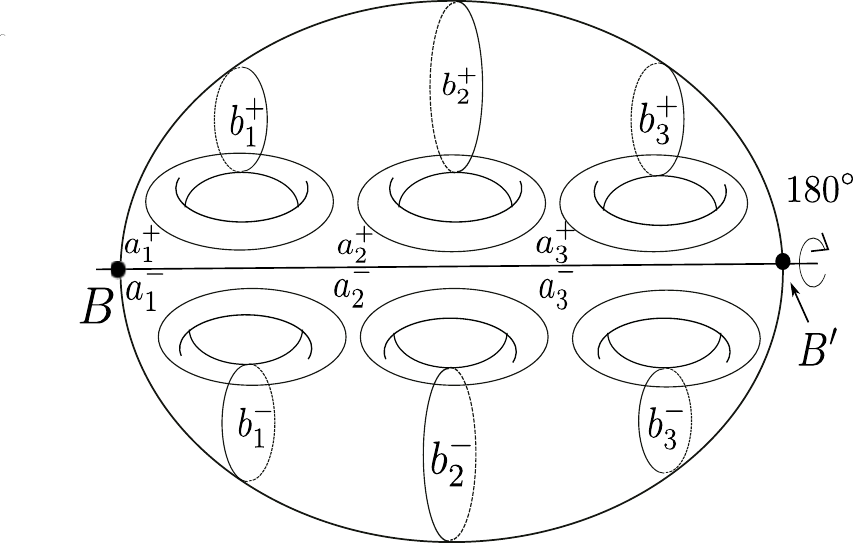}
 \caption{Deck transformation $\sigma$}
  \label{figure4}
\end{figure}

 Let PMod$_{6,2}^{\sigma}$ be the centralizer of $\sigma$ in PMod$_{6,2}$. We have a map 
$p_\sigma:  \text{PMod}_{6,2}^{\sigma}\to \text{Mod}_{3,2}$. By the construction, $\pi_1(S_{129})$ acts trivially on $H^1(S_{3,2};\mathbb{Z}/2)$; it can also be seen by computing the $\rho_{\tau}(\pi_1(S_3))$-action on $H^1(S_{3,2};\mathbb{Z}/2)$. The monodromy $\rho_\tau':=\rho_\tau |_{\pi_1(S_{129})}:\pi_1(S_{129})\to \pi_1(S_3)\to \text{Mod}_{3,2}$ admits a lift to $\text{PMod}_{6,2}^{\sigma}$ as the following diagram \eqref{LD}; i.e. there exists $\rho$ such that $p_\sigma \circ \rho=\rho_\tau'$. Let $f: \text{PMod}_{6,2}\to \text{Mod}_6$ be the forgetful map and let $\rho_{AK}=\rho\circ f$ be the monodromy representation of a lift.
\begin{equation}
\xymatrix{
&\text{PMod}_{6,2}^{\sigma}\ar[d]^{p_\sigma} \ar[r]^f&\text{Mod}_6\\
\pi_1(S_{129})\ar[r]_{\rho_\tau'}\ar@{-->}[ur]^-\rho \ar@{-->}[urr]^{\rho_{AK}}& \text{Mod}_{3,2}.&\\
}\label{LD}
\end{equation}
The geometric construction depends on some noncanonical parameters; similarly this phenomenon reappears when we consider the monodromy representation. This will be discussed in the following remark.

\begin{rmk}
The lift $\rho$ in diagram \eqref{LD} is not unique! Let $\{g_i, h_i\}$ be the generators of $\pi_1(S_{129})$ such that $\pi_1(S_{129})=\langle g_i,h_i| \prod_i[g_i,h_i]=1\rangle$. Because $\sigma$ commutes with any element in the set $\{G_i=\rho_{AK}(g_i), H_i=\rho_{AK}(h_i)\}$, we could multiply $\sigma$ with a subset of $\{G_i, H_i\}$ to obtain a new monodromy representation. For example, $\{G_i\sigma, H_i\}$ is a new monodromy representation.
Among all the different monodromy representations, are the total spaces of the surface bundles diffeomorphic to each other? 
\end{rmk}

\section{The proof of Theorem \ref{main2}}
In this section, we will give a proof of Theorem \ref{main2} by computing $H^1(M_{AK};\mathbb{Q})$.

\subsection{The lift of a square of a point push}
In this subsection, we will determine the lifts of some elements of $\pi_1(S_{129})$ to Mod$_6$ under the branched cover. For any simple closed curve $c$, we denote the Dehn twist about $c$ by $T_c$. For any loop $L$ at the base point $B$, denote  the point-pushing map on $L$ by Push$(L)$. Let $a$ be the loop in Figure \ref{figure2}. We have Push$(a)=T_xT_y^{-1}$; e.g. see \cite[Fact 4.7]{BensonMargalit}.

\begin{figure}[H]
\minipage{0.48\textwidth}
  \includegraphics[width=\linewidth]{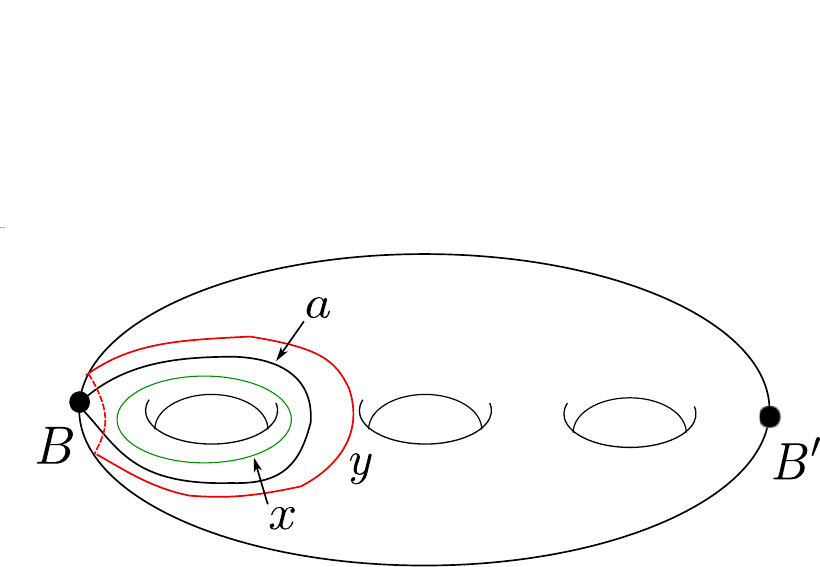}
  \caption{Point-pushing. The lighter colors (green and red) represent $x$ and $y$
  and the darker color (black) represents $a$}
  \label{figure2}
\endminipage\hfill
\minipage{0.48\textwidth}
  \includegraphics[width=\linewidth]{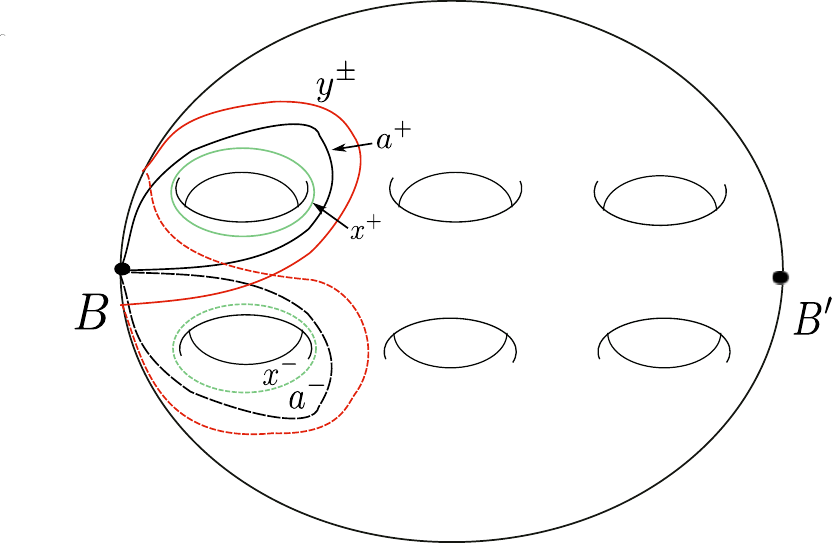}
  \caption{Lifts of $a,x,y$ with the same colorings as Figure \ref{figure2}}
\label{figure3}
\endminipage\hfill
\end{figure}

Since $B$ is one of the branched points of the $\mathbb{Z}/2$-cover $\pi: S_{6}\to S_{3}$, one of the curves $x$ or $y$ will lift to two copies and the other will lift to a single copy. The curve $a$ will have two lifts, which we call $a^-$ and $a^+$. Since $\text{Push}(a)^2$ acts trivially on $H_1(S_{3,2};\mathbb{Z}/2)$, the action $\text{Push}(a)^2$ lifts to an action on $S_6$. Let Lift$(\text{Push}(a)^2)$ be the lift of the point-pushing action on $S_6$. For any two curves $c_1,c_2$, let $i(c_1,c_2)$ be the algebraic intersection number of $c_1$ and $c_2$. In the following computation, we will use the letter $a,x,y$ to represent either the curves or the homology classes that the curves represent.
\begin{lem}
Pick a direction on $a$ and assign directions on $a^+,a^-$ accordingly. For any $c\in H_1(S_6;\mathbb{Q})$, the action of Lift$(\text{Push}(a)^2)$ on $c$ has the following 2 possibilities:
\[\text{Lift}(\text{Push}(a)^2)(c)=c\pm i(c,a^+-a^-)(a^+-a^-)\]
or
\[\text{Lift}(\text{Push}(a)^2)(c)=\sigma_*(c\pm i(c,a^+-a^-)(a^+-a^-)).\]
\label{3.1.2}
\end{lem}
\begin{proof}
Suppose without loss of generality that Lift$(x)=x^+\cup x^-$ and Lift$(y)=y^\pm$. By looking at the action locally, we have that Lift$(T_x^2)=T_{x^-}^2T_{x^+}^2$ and Lift$(T_y^2)=T_{y^\pm}$. Therefore
\[
\text{Lift}(\text{Push}(a)^2)=\text{Lift}(T_x^2T_y^{-2})=T_{x^-}^2T_{x^+}^2T_{y^\pm}^{-1}.
\]
$x^+$ and $x^-$ are homotopic to $a^+$ and $a^-$ on $S_6$ respectively, so as homology classes $x^+=a^+$ and $x^-=a^-$. Since $x^+,x^-,y^\pm$ bound a pair of pants, there exists an orientation of $y^\pm$ such that as a homology class, $y^\pm=a^++a^-$. Thus we have the following computation on the action of the homology.
\begin{equation}
\begin{split}
T_{x^+}^2T_{x^-}^2T_{y^\pm}^{-1}(c) & =c-i(c,y^\pm)y^\pm+i(c,x^+)2x^++i(c,x^-)2x^-\\
&=c-i(c,a^++a^-)(a^++a^-)+i(c,a^+)2a^++i(c,a^-)2a^-\\
&=c+i(c,a^--a^+)(a^--a^+).
\end{split}
\end{equation}

In the case where Lift$(y)=y^+\cup y^-$ and Lift$(x)=x^\pm$, we have 

\[\text{Lift}(\text{Push}(a)^2)(c)=c-i(c,a^+-a^-)(a^+-a^-).\]

Since every element has two lifts that differ by the deck transformation $\sigma$, we have the second possibility.
 \end{proof}

\subsection{\boldmath  The eigendecomposition of the action of $\sigma_*$ on $H^1(S_6;\mathbb{Q})$}
In this subsection, we will discuss the eigendecomposition of the action of $\sigma_*$ on $H^1(S_6;\mathbb{Q})$ and determine the image of $\pi^*:H^1(S_3;\mathbb{Q})\to H^1(S_6;\mathbb{Q})$ induced from the $\mathbb{Z}/2$-branched cover $\pi:S_6\to S_3$. The action of the deck transformation $\sigma$ on $S_6$ induces a decomposition of $H_1(S_6;\mathbb{Q})$ by the eigenvalues of the $\sigma_*$-action on $H_1(S_6;\mathbb{Q})$. Since $\sigma$ is an involution, the eigenvalues of $\sigma_*$ are $\{\pm1\}$. Let $H^+$ be the eigenspace of $\sigma$ associated with eigenvalue $+1$ and $H^-$ the eigenspace of $\sigma$ associated with eigenvalue $-1$. Then there is a direct sum
\[ 
H_1(S_6;\mathbb{Q})=H^-\oplus H^+.
\]
Via the universal coefficient theorem, $f\in H^1(S_6;\mathbb{Q})$ corresponds to a functional $f:H_1(S_6;\mathbb{Q})\to \mathbb{Q}$.
\begin{claim}
A functional $f:H_1(S_6;\mathbb{Q})\to \mathbb{Q}$ belongs to $\pi^*H^1(S_3;\mathbb{Q})$ if and only if $H^{-} \subset ker(f)$.
\end{claim}
\begin{proof}
It is classical that $\pi^*:H^1(S_3;\mathbb{Q})\to H^1(S_6;\mathbb{Q})^{\mathbb{Z}/2}$ is an isomorphism; e.g. see \cite[Theorem III.2.4]{MR0413144}. Then $f\in H^1(S_6;\mathbb{Q})^{\mathbb{Z}/2}$ if and only if $\sigma^*(f)=f$ if and only if $\sigma^*(f)(x)=f(x)$ for any $x\in H_1(S_6;\mathbb{Q})$ if and only if $f(\sigma_*(x)-x)=0$ for any $x\in H_1(S_6;\mathbb{Q})$. Since $\sigma_*$ is an involution, we know that the subspace of $H_1(S_6;\mathbb{Q})$ that are spanned by $\{\sigma_*(x)-x|x\in H_1(S_6;\mathbb{Q})\}$ is $H^-$. Therefore $f\in \pi^*H^1(S_3;\mathbb{Q})$ if and only if $H^{-} \subset ker(f)$.
\end{proof}

In Figure \ref{figure4}, we have a geometric description of a basis $\{a_1^+,a_1^-,...\}$ of $H_1(S_6;\mathbb{Q})$ where $a_i^+$ and $a_i^-$ or $b_i^+$ and $b_i^-$ are each other's images under the $\sigma_*$ action. 

\subsection{\boldmath The $\pi_1(S_{129})$-invariant cohomology}
Let $a_1$ be a simple loop based at $B$ as in Figure \ref{figure10}. Since $a_1$ does not intersect $\tau({a_1})$, the monodromy action of $a_1$ on $S_{3,2}$ is the product of point-pushings at $B$ and $B'$ of $a_1$ and $\tau(a_1)$. By the monodromy description of $M_{AK}$, we have that $\rho_{AK}(a_1^2)=\text{Lift}(\text{Push}(a_1)^2\text{Push}(\tau(a_1))^2)$. 

\begin{lem}
Let $f\in H^1(S_6;\mathbb{Q})^{\pi_1(S_{129})}$ be an invariant cohomology class. Then $f$ satisfies $f(a_1^+-a_1^-)=0$ and $f(a_3^+-a_3^-)=0$.
\label{comp}
\end{lem}
\begin{proof}
By Lemma \ref{3.1.2}, $\text{Push}(a_1)^2\text{Push}(\tau(a_1))^2$ has two possible lifts that differ by the deck transformation $\sigma$. 
\\
{\bf Case 1:} For any $c\in H_1(S_6;\mathbb{Q})$,
\[
\rho_{AK}(a_1^2)(c)=
c\pm i(c,a_1^+-a_1^-)(a_1^+-a_1^-)\pm i(c,a_3^+-a_3^-)(a_3^+-a_3^-).
\]
Since $f$ is invariant under the action of $\rho_{AK}(a_1^2)$, we have $f(\rho_{AK}(a_1^2)(c))=f(c)$ for any $c\in H_1(S_6;\mathbb{Q})$. After evaluating $f$ on both sides, we obtain:

\[
f(c)=f(c)\pm i(c,a_1^+-a_1^-)f(a_1^+-a_1^-)\pm i(c,a_3^+-a_3^-)f(\tilde{a_3}-a_3^+).
\]  
Equivalently, 
\[
i(c,a_1^+-a_1^-)f(a_1^+-a_1^-)\pm i(c,a_3^+-a_3^-)f(a_3^+-a_3^-)=0.
\]
However, $a_1^+-a_1^-$ and $a_3^+-a_3^-$ are independent elements in $H_1(S_6;\mathbb{Q})$, so we can find $c$ such that $i(c,a_1^+-a_1^-)=0$ and $i(c,a_3^+-a_3^-)=1$. Therefore we must have $f(a_3^+-a_3^-)=0$. By the same argument, $f(a_1^--a_1^+)=0$.
\\
\\
{\bf Case 2:} For any $c\in H_1(S_6;\mathbb{Q})$, \[
\rho_{AK}(a_1^2)(c)=
\sigma_*(c\pm i(c,a_1^+-a_1^-)(a_1^+-a_1^-)\pm i(c,a_3^+-a_3^-)(a_3^+-a_3^-)).
\]
Since $f$ is invariant under the action of $a_1^2$, we have $f(\rho_{AK}(a_1^2)(c))=f(c)$. After evaluating $f$ on both sides, we obtain:
\[f(c)=f(\sigma_*(c))\pm i(c,a_1^--a_1^+)f(\sigma_*(a_1^--a_1^+))\pm i(c,a_3^+-a_3^-)f(\sigma_*(a_3^+-a_3^-)).
\] 
If we set $c=a_1^-$ and $c=a_3^-$, we obtain 
\[
f(a_1^-)=f(a_1^+)
\]
and 
\[
f(a_3^-)=f(a_3^+).
\]
\end{proof}
This allows us to determine the full invariant cohomology $H^1(S_6;\mathbb{Q})^{\pi_1(S_{129})}$.
\begin{lem}
We have the following isomorphism \label{keypoint}
\[
H^1(S_6;\mathbb{Q})^{\pi_1(S_{129})}\cong H^1(S_3;\mathbb{Q}).
\]
\end{lem}
\begin{proof}
By using the same argument as in Lemma \ref{comp} on $(b_1)^2$ and $(b_2a_1)^2$, we obtain that $f(b_3^+-b_3^-)=0$, $f(b_1^+-b_1^-)=0$ and 
\[
f((b_2+a_1)^+-(b_2+a_1)^-)=0.
\]
Since we already have $f(a_1^+-a_1^-)=0$, we obtain that $f(b_2^+-b_2^-)=0$.

From the above discussion, any $f\in H^1(S_6;\mathbb{Q})^{\pi_1(S_{129})}$ is zero on the 5-dimensional space spanned by $a_1^+-a_1^-$, $b_1^+-b_1^-$ , $b_3^+-b_3^-$ , $b_3^+-b_3^-$ , $b_2^+-b_2^-$. Therefore dim$\left(H^1(S_6;\mathbb{Q})^{\pi_1(S_{129})}\right)\le 7$. Since $\pi^*H^1(S_3;\mathbb{Q}))\subset H^1(S_6;\mathbb{Q})^{\pi_1(S_{129})}$, we have that dim$\left(H^1(S_6;\mathbb{Q})^{\pi_1(S_{129})}\right)\ge 6$. Via the Serre spectral sequence of the fiber bundle $S_{6}\to M_{AK}\to S_{129}$, we have
\[
1\to H^1(S_{129};\mathbb{Q})  \to H^1(M_{AK};\mathbb{Q})\to H^1(S_6;\mathbb{Q})^{\pi_1(S_{129})} \to 1.
\]
This implies
\[
\text{dim}\left(H^1(M_{AK};\mathbb{Q})\right)=\text{dim}\left( H^1(S_{129}\right);\mathbb{Q})+\text{dim}\left(H^1(S_6;\mathbb{Q})^{\pi_1(S_{129})}\right) .
\]
As a branched cover of an algebraic surface along an algebraic 
curve, the manifold $M_{AK}$ is itself an algebraic surface; thus $M_{AK}$ is a K\"ahler manifold. Therefore, dim$\left(H^1(M_{AK};\mathbb{Q})\right)$ must be an even number. So we have dim$\left(H^1(S_6;\mathbb{Q})^{\pi_1(S_{129})}\right)= 6$, which shows
\[
H^1(S_6;\mathbb{Q})^{\pi_1(S_{129})}\cong H^1(S_3;\mathbb{Q}).
\]
\end{proof}

\subsection{Finishing the proof of Theorem \ref{main2}}
We will now finish the proof of Theorem \ref{main2}.
\begin{proof}[Proof of Theorem \ref{main2}]
Since $M_{AK}\xrightarrow{P} S_{129}\times S_3$ is a finite branched cover, the induced map $P^*:H^4(S_3\times S_{129};\mathbb{Q})\to H^4(M_{AK};\mathbb{Q})$ is an isomorphism. By Poincar\'e duality, $P^*:H^k(S_3\times S_{129};\mathbb{Q})\to H^k(M_{AK};\mathbb{Q})$ is injective for any $k$. Via Lemma \ref{keypoint}, we also have 
\[H^1(M_{AK};\mathbb{Q})\cong H^1(S_3;\mathbb{Q})\oplus H^1(S_{129};\mathbb{Q}).\]

Therefore $M_{AK}$ satisfies the assumption of Lemma \ref{main1}, which shows SFib$(M_{AK})=2$.
\end{proof}

\section{Fibering number of a finite cover of a trivial bundle}
In this section, we will prove Theorem \ref{thm3}. Let $E$ be a regular finite cover of $B\times F$ where $B$ and $F$ are surfaces with $g(B)>1$ and $g(F)>1$. Let $p_1: E\to B$ and $p_2:E\to F$ be the projections. Denote the image of $p_{1*}:\pi_1(E)\to \pi_1(B)$ by Im$(p_1)$ and the image of $p_{2*}:\pi_1(E)\to \pi_1(F)$ by Im$(p_2)$.
\begin{lem}
With the above definitions and assumptions, we have the following
\[
H^1(E;\mathbb{Q})\cong H^1(\text{Im}(p_1);\mathbb{Q})\oplus H^1(\text{Im}(p_2);\mathbb{Q}).
\]
\end{lem}
\begin{proof}
All spaces involved here are $K(\pi,1)$ spaces, sometimes we transition between group cohomology and cohomology of spaces. Let $\pi_1(\tilde{F})$ be the kernel of $p_{1*}:\pi_1(E)\to \text{Im}(p_1)$. Since $\pi_1(E)$ is a finite index normal subgroup of Im$(p_1)\times \text{Im}(p_2)$, we have that $\pi_1(\tilde{F})$ is a finite index normal subgroup of $\text{Im}(p_2)$. We have the following commutative diagram.
\begin{equation}
\xymatrix{
1\ar[r]& \pi_1(\tilde{F}) \ar[r]\ar[d]& \pi_1(E)\ar[r]\ar[d]& \text{Im}(p_1)\ar[r]\ar[d]^=& 1\\
1\ar[r]&\text{Im}(p_2)\ar[r]&  \text{Im}(p_1)\times  \text{Im}(p_2)\ar[r]&  \text{Im}(p_1)\ar[r] & 1.
}
\label{ex1}
\end{equation}
The group $\pi(E)$ acts on $H^1(\tilde{F};\mathbb{Q})$ by conjugacy. Since Im$(p_1)$ is commutative with Im$(p_2)$ and $p_{1*}:\pi_1(E)\to \text{Im}(p_2)$ is surjective, we have that
 \[
 H^1(\tilde{F};\mathbb{Q})^{\text{Im}(p_2)}\cong H^1(\tilde{F};\mathbb{Q})^{\pi_1(E)}.
 \]
Since $\pi_1(\tilde{F})$ is a finite index subgroup of $\text{Im}(p_2)$, we have that 
\[
H^1(\tilde{F};\mathbb{Q})^{\text{Im}(p_2)}\cong H^1(\text{Im}(p_2);\mathbb{Q}).
\]

By the top exact sequence of (\ref{ex1}), we obtain the following
\[
0\to H^1(\text{Im}(p_1);\mathbb{Q})\to H^1(E;\mathbb{Q})\to H^1(\tilde{F};\mathbb{Q})^{\pi(E)}\cong H^1(\text{Im}(p_2);\mathbb{Q})\to 0.
\]
The lemma follows.
\end{proof}

\begin{proof}[Proof of Theorem \ref{thm3}]
Since $\pi_1(E)$ is a finite index subgroup of $\text{Im}(p_1)\times  \text{Im}(p_2)$, we obtain that $H^4(\text{Im}(p_1)\times  \text{Im}(p_2);\mathbb{Q})\to H^4(E;\mathbb{Q})$ is an isomorphism. By Poincar\'e duality, $H^k(\text{Im}(p_1)\times  \text{Im}(p_2);\mathbb{Q})\to H^k(E;\mathbb{Q})$ is injective for every $k$. More specifically, $H^2(\text{Im}(p_1)\times  \text{Im}(p_2);\mathbb{Q})\subset H^2(E;\mathbb{Q})$. Therefore $E$ satisfies the assumptions of Lemma \ref{main1}, which shows SFib$(E)=2$.

\end{proof}

\section{An example with exactly 4 fiberings}
Now we deal with an example of Salter \cite{salter2015} and we prove that it has exactly 4 fiberings. As we mentioned before, the equivalence relation we choose is $\pi_1$-fiberwise diffeomorphism not fiberwise diffeomorphism. Under fiberwise diffeomorphism, $M_S$ only has one fibering. 

Let $\triangle$ be the diagonal in $S_g\times S_g$. Let $M_S=(S_g\times S_g-\triangle) \cup_{\theta} (S_g\times S_g-\triangle)$, where $\theta$ is the identification of the boundaries of the two copies of $S_g\times S_g-\triangle $. Each copy of $S_g\times S_g-\triangle$ has two fiberings $p_1$ and $p_2$ where $p_i$ is the projection onto the $i$th coordinate. Therefore $M_S$ has four obvious fiberings: $\{(p_i,p_j) \text{ for } i,j=1,2\}$.

There is a subtlety in defining $(p_1,p_2)$ and $(p_2,p_1)$ in the smooth category, but the details will be immaterial here. See \cite[Section 2]{salter2015}.

\begin{lem}\label{lem}
With the notation as in the previous paragraph and for $g\ge 2$, 
\[H^1(S_g\times S_g-\triangle;\mathbb{Q})\cong p_1^*(H^1(S_g;\mathbb{Q}))\oplus p_2^*(H^1(S_g;\mathbb{Q})).\]
\end{lem}
\begin{proof}
By the Thom isomorphism theorem, $H^1(S_g\times S_g,S_g\times S_g-\triangle;\mathbb{Q})=0$. The following long exact sequence of relative cohomology
\[
0\to H^1(S_g\times S_g;\mathbb{Q})\to H^1(S_g\times S_g-\triangle;\mathbb{Q})\to H^1(S_g\times S_g,S_g\times S_g-\triangle;\mathbb{Q})=0\]
tells us that $H^1(S_g\times S_g;\mathbb{Q})\cong H^1(S_g\times S_g-\triangle;\mathbb{Q})$. By the K\"unneth formula, \[H^1(S_g\times S_g-\triangle;\mathbb{Q})\cong p_1^*(H^1(S_g;\mathbb{Q}))\oplus p_2^*(H^1(S_g;\mathbb{Q})).\]
This completes the proof.
\end{proof}
Let 
\[
add: H^1(S_g;\mathbb{Q})\oplus H^1(S_g;\mathbb{Q})\oplus H^1(S_g;\mathbb{Q})\oplus H^1(S_g;\mathbb{Q})\to H^1(S_g; \mathbb{Q})
\]
be the addition of elements in the abelian group $H^1(S_g; \mathbb{Q})$.
\begin{lem} 
For $g>1$, we have the following exact sequences:
\begin{equation}
0\to H^1(M_S;\mathbb{Q}) \xrightarrow{E^1} H^1(S_g;\mathbb{Q})\oplus H^1(S_g;\mathbb{Q})\oplus H^1(S_g;\mathbb{Q})\oplus H^1(S_g;\mathbb{Q})\xrightarrow{add} H^1(S_g;\mathbb{Q})\to 0
\label{iden}
\end{equation}
and 
\[
0\to H^2(M_S;\mathbb{Q}) \xrightarrow{E^2} H^2(S_g\times S_g-\triangle;\mathbb{Q})\oplus H^2(S_g\times S_g-\triangle;\mathbb{Q}).
\]
\label{computation}
\end{lem}

\begin{proof}
Let $M_1$ and $M_2$ be the two copies of $S_g\times S_g-\triangle$ in the construction of $M_S$. Define $N:=M_1\cap M_2$; this is a circle bundle over $S_g$. The circle bundle $N$ has Euler number $2-2g\neq 0$, by the Serre spectral sequence of the circle bundle $S^1\to N\to S_g$, we have 
\[H^1(N;\mathbb{Q})=H^1(S_g;\mathbb{Q}).
\]

The map 
\[
H_1(N;\mathbb{Q})=H_1(S_g;\mathbb{Q})\to H_1(S_g\times S_g;\mathbb{Q})=H_1(S_g;\mathbb{Q})\oplus H_1(S_g;\mathbb{Q})
\]
 is induced by the diagonal embedding. Therefore 
\[
 H^1(S_g\times S_g;\mathbb{Q})=H^1(S_g;\mathbb{Q})\oplus H^1(S_g;\mathbb{Q})\to H^1(N;\mathbb{Q})=H^1(S_g;\mathbb{Q})
\]
  is the addition of the two elements (dual to the diagonal map). Consequently we have a long exact sequence coming from the Mayer--Vietoris pair $(M_1,M_2)$:
\begin{equation}
\begin{split}
0\to & H^1(M_S;\mathbb{Q}) \xrightarrow{E^1}H^1(S_g\times S_g-\triangle;\mathbb{Q})\oplus H^1(S_g\times S_g-\triangle;\mathbb{Q})\xrightarrow{s^*} H^1(N;\mathbb{Q})\to  \\
\to &H^2(M_S;\mathbb{Q}) \xrightarrow{E^2} H^2(S_g\times S_g-\triangle;\mathbb{Q})\oplus H^2(S_g\times S_g-\triangle;\mathbb{Q}).
\end{split}
\end{equation}
We know $s^*$ is surjective, therefore $E^1$ and $E^2$ are injective.
\end{proof} 

By the short exact sequence (\ref{iden}), we identify $H^1(M_S;\mathbb{Q})$ as a subspace of $H^1(S_g;\mathbb{Q})\oplus H^1(S_g;\mathbb{Q})\oplus H^1(S_g;\mathbb{Q})\oplus H^1(S_g;\mathbb{Q})$.
We use $x,y,z,w$ to represent the coordinates. Therefore any element $a\in H^1(M_S;\mathbb{Q})$ can be written as $a=(x,y,z,w)\in H^1(S_g;\mathbb{Q})\oplus H^1(S_g;\mathbb{Q})\oplus H^1(S_g;\mathbb{Q})\oplus H^1(S_g;\mathbb{Q})$ such that $x+y+z+w=0\in H^1(S_g;\mathbb{Q})$. We also identify $H^1(S_g\times S_g-\triangle;\mathbb{Q})$ with $H^1(S_g;\mathbb{Q})\oplus H^1(S_g;\mathbb{Q})$ by Lemma \ref{lem}. Any element $a'\in H^1(S_g\times S_g-\triangle;\mathbb{Q})$ can be written as $a'=(x,y)\in H^1(S_g;\mathbb{Q})\oplus H^1(S_g;\mathbb{Q})$.

We will need the following algebraic lemma \cite[ Lemma 3.7]{chen2016universal} on the cup product structure of $H^*(S_g\times S_g-\triangle;\mathbb{Q})$.
\begin{lem}
Let $r,s \in H^1(S_g\times S_g-\triangle;\mathbb{Q})$ be two independent elements. If $r\smile s=0$, then $r,s\in p_i^*(H^1(S_g;\mathbb{Q}))$ for some $i\in \{1,2\}$.
\label{pc}
\end{lem}

\begin{proof}[Proof of Theorem \ref{thm5}]

From the naturality of cup product, we have the following commutative diagram:
\begin{equation}
\xymatrix{
\Lambda^2E^1:\Lambda^2 H^1(M_S;\mathbb{Q})\ar[r] \ar[d]^{\text{cup}}& \Lambda^2( H^1(S_g\times S_g-\triangle;\mathbb{Q})\oplus H^1(S_g\times S_g-\triangle;\mathbb{Q})) \ar[d]^{\text{cup}}\\
E^2:H^2(M_S;\mathbb{Q})\ar[r] &H^2(S_g\times S_g-\triangle;\mathbb{Q})\oplus H^2(S_g\times S_g-\triangle;\mathbb{Q}).
}
\label{natural}
\end{equation}
\indent
Let $S_h\to E\xrightarrow{p} B$ be some fibering. Since $\chi(M_S)>0$ and $\chi(S_h)<0$, by computation $\chi(B)<0$ implying that $g(B)>1$. Define $H:=p^*(H^1(B;\mathbb{Q}))$. Since $g(B)>1$, there exists linearly independent $b,b'\in H^1(B;\mathbb{Q})$ so that $b \smile b'=0$. Let $p^*(b)=(x,y,z,w), p^*(b')=(x',y',z',w')\in H$. By Lemma \ref{computation} and diagram (\ref{natural}), we have that $p^*(b)\smile p^*(b')=0$ if and only if $(x,y)\smile (x',y')=0\in H^2(S_g\times S_g-\triangle;\mathbb{Q})$ and $(z,w)\smile (z',w')=0\in H^2(S_g\times S_g-\triangle;\mathbb{Q})$. By Lemma \ref{pc}, we have the following possibilities: one of (1) and $(1'$) must be true and one of of (2) and $(2'$) must be true.
\begin{itemize}[leftmargin=*]
\item[] (1) $(x,y)$ and $(x',y')$ are dependent in $H^1(S_g\times S_g-\triangle;\mathbb{Q})$.
\item[]($1'$)  $x=x'=0$ or $y=y'=0$.
\item[] (2) $(z,w)$ and $(z',w')$ are dependent in $H^1(S_g\times S_g-\triangle;\mathbb{Q})$.
\item[] $(2'$) $z=z'=0$ or $w=w'=0.$
\end{itemize}
In the original four fiberings, the subspaces $\{(p_i,p_j)^*(H^1(S_g;\mathbb{Q}))\text{ for } i,j=1,2\}$ satisfy the following.
\begin{itemize}
\item
$(p_1,p_1)^*(H^1(S_g;\mathbb{Q}))$ contains all elements that $y=0$ and $w=0$.
\item
$(p_1,p_2)^*(H^1(S_g;\mathbb{Q}))$ contains all elements that $y=0$ and $z=0$.
\item
$(p_2,p_1)^*(H^1(S_g;\mathbb{Q}))$ contains all elements that $x=0$ and $w=0$.
\item
$(p_2,p_2)^*(H^1(S_g;\mathbb{Q}))$ contains all elements that $x=0$ and $z=0$.
\end{itemize}
If ($1'$) and ($2'$) are true for $p^*(b)$, then $p^*(b)$ belongs to one of the four spaces above. By Lemma 3.6, the fibering $S_h\to E\xrightarrow{p} B$ must be one of the four original fiberings. Thus to conclude the proof of Theorem \ref{thm5}, it suffices to prove the following Claim \ref{H}.

\begin{claim}
There exists an element  in the subspace H satisfying $(1')$ and $(2')$.
\label{H}
\end{claim}
We now prove the claim. We assume that no element in $H$ satisfies ($1')$ and ($2')$. Since $g(B)>1$, for any element $a\in H$, the dimension of the subspace $\{h\in H:a\smile h=0\}$ is at least 3. We break our discussion into three cases.
\begin{itemize}
\item
{\bf\boldmath Case 1: There is an element $a=(x,y,z,w)\in H$ such that $x$, $y$, $z$ and $w$ are all nonzero.} Find $b\in H$ such that $b\smile a=0$. Via Lemma \ref{pc} ,we have that $a=(kx,ky,lz,lw)$ for $k,l\in \mathbb{Q}$. However the subspace $\{(kx,ky,lz,lw):k,l\in \mathbb{Q}\}$ is only 2-dimensional; this contradicts the fact that the dimension of the subspace $\{h\in H:a\smile h=0\}$ is at least 3. 
\item
{\bf\boldmath Case 2: There exists an element $a=(x,y,z,w)\in H$ such that $x=y=0$ and $z,w\neq 0$.} We know that $z+w=0$ by Lemma \ref{computation}. Let $b=(x',y',z',w')\in H$. If $b \smile a=0$, then $(z',w')=k(z,w)$ for $k\in \mathbb{Q}$ by Lemma \ref{pc}. The dimension of the space $\{(0,0,mz,mw):m\in \mathbb{Q}\}$ is 1 so there exists $b=(x',y',kz,kw)\in \{h\in H:a\smile h=0\}$ such that $x'$ or $y'$ nonzero. Since $z+w=0$, we have that $x'+y'=0$. Thus $x'$ and $y'$ are both nonzero. Let $l\in \mathbb{Q}$ such that $l+k\neq 0$. The linear combination $la+b=(x',y',(l+k)z,(l+k)w)\in H$ has all coordinates nonzero; this reduces to Case 1.
\item
{\bf\boldmath Case 3: Every nonzero element $(x,y,z,w)\in H$ has exactly one coordinate equal to zero.} If two elements $a,b\in H$ have different coordinates zero, we could find a linear combination $ka+lb\in H$ for $k,l\in \mathbb{Q}$ such that all coordinates of $ka+lb$ are nonzero; this reduces to Case 1. Therefore all elements in $H$ have the same coordinate equal to zero.

Assume without loss of generality that every element $(x,y,z,w)\in H$ only has $w=0$. There are independent elements $a=(x,y,z,0),b=(x',y',z',0)\in H$ such that $a\smile b=0$. By Lemma \ref{computation} and Lemma \ref{pc}, we have $(x',y')=k(x,y)$ for $k\in \mathbb{Q}$. Since $a,b$ are independent, we know that $z'\neq kz$. Then the nonzero element $k(x,y,z,0)-(x',y',z',0)=(0,0,kz-z',0)\in H$ only has one coordinate nonzero. This is a contradiction to the assumption of Case 3.
\end{itemize}
This completes the proof of the claim, hence the theorem.
\end{proof}

\bibliographystyle{alpha}
\bibliography{citing}{}

\end{document}